\documentclass[12pt]{amsart}
 
\usepackage{amsthm,amssymb,amsmath,enumerate,graphicx, tikz}
\usepackage{amscd}
\usepackage{setspace}
\theoremstyle{plain}
\usepackage[numbers, square]{natbib}

\newtheorem{theorem}{Theorem}[section]

\newtheorem{lemma}[theorem]{Lemma}

\theoremstyle{definition}
\newtheorem{definition}[theorem]{Definition}

\theoremstyle{remark}

\newcommand{\cf}{\mathcal{F}}
\newcommand{\F}{\mathcal{F}}
\newcommand{\C}{\mathcal{C}}

\newcommand{\vx}{f}

\newcommand{\M}{\mathcal{M}}

\newcommand{\cx}{\mathcal{X}}

\newcommand{\R}{\mathbb{R}}

\newcommand{\vab}{\nu^*_{\a,\b}}
\newcommand{\tab}{\tau^*_{\a,\b}}

\newcommand{\bk}{\overline{k}}

\renewcommand{\a}{\mathbf{a}}
\renewcommand{\b}{\mathbf{b}}
\renewcommand{\c}{\mathbf{c}}
\renewcommand{\r}{\mathbf{r}}
\newcommand{\bb}{\underline{b}}
\newcommand{\ba}{\underline{a}}

\newcommand{\one}{\mathbf{1}}
\newcommand{\rk}{\textup{rk}}
\newcommand{\rank}{\textup{rk}}
\newcommand{\supp}{\text{supp}}

\title{Choice functions in the intersection of matroids}

\date{\today}

\author{Joseph Briggs}
\address{Department of Mathematics,
Technion\\
Haifa, Israel} \email{briggs@campus.technion.ac.il}

\author{Minki Kim}
\address{Department of Mathematics,
Technion\\
Haifa, Israel} \email{kimminki@campus.technion.ac.il}

\begin{document}

\maketitle

\begin{abstract}
    We prove a common generalization of two results, one on rainbow fractional matchings \cite{AHJ} and one on rainbow sets in the intersection of two matroids \cite{KZ}:
    Given $d=r\lceil k\rceil -r+1$ functions of size (=sum of values) $k$ that are all independent in each of $r$ given matroids, there exists a rainbow set of $supp(f_i), ~i \le d$, supporting a function with the same properties. 
\end{abstract}

\section{Introduction}

Let $\cf=(F_1, \ldots ,F_m)$ be a family (namely, a multiset) of  sets.  A (partial) {\em rainbow set} for $\cf$ is 
the image of a partial choice function.
Namely, it is a
set of the form $R = \{x_{i_1}, x_{i_2},\ldots ,x_{i_k}\} $, where $1 \le i_1<i_2< \ldots <i_k\le k$, and $x_{i_j}\in F_{i_j} ~~(j \le k)$. Here it is assumed that $R$ is a set, namely that the elements $x_{i_j}$ are distinct. There are many theorems of the form ``under some conditions there exists a rainbow set  satisfying a prescribed condition''. For example, the case where the condition is being full (representing all $F_i's$) is the subject of Hall's marriage theorem. The following theorem of Aharoni and Berger \cite{AB}, which generalizes a result of Drisko \cite{drisko}, belongs to this family, and is a forefather of the results in the present paper:

\begin{theorem}\label{t.drisko}
Any family of $2k-1$  matchings of size $k$ in a bipartite graph $G$ have a rainbow matching of size $k$.
\end{theorem}

(Drisko's slightly narrower result was formulated in the language of  Latin rectangles.) In \cite{gendrisko} it was conjectured that almost the same is true in general graphs, namely that in any graph $2k$ matchings of size $k$ have a rainbow matching of size $k$,  and that for odd $k$ the Drisko bound suffices - $2k-1$ matchings of size $k$ have a rainbow matching of size $k$. This is far from being solved (in \cite{gendrisko} the bound $3k-2$ was proved), but in \cite{AHJ} a fractional version of the conjecture was proved, in a more general setting. Recall that $\nu^*(F)$ denotes the largest total weight of a fractional matching in a hypergraph $H$.

 \begin{theorem}[Aharoni, Holzman and Jiang \cite{AHJ}]\label{t.hypermatchings}
 Let $m$ be a real number,  let  $H$ be an $r$-uniform hypergraph and let $q \ge {\lceil rk\rceil }$ be an integer. Then any family $E_1,...,E_q$ of sets of edges in $H$  satisfying $\nu^*(E_j) \geq k$ for all
 $j \le q$ has a rainbow set $F$ of edges with $\nu^*(F) \ge k$. If $H$ is $r$-partite then it suffices to assume that $q \ge r \lceil k\rceil-r+1$ to obtain the same conclusion. 
 
 \end{theorem}
 
 Drisko's theorem is a special case, since in bipartite graphs $\nu^*=\nu$. The integral version of the theorem is false for $r>2$. 
 For $r=3, k=2$, for example, $rk-r+1=4$, and the four matchings of size $2$ in the complete  $2 \times 2 \times 2$ $3$-partite hypergraph do not have a rainbow matching of size $2$,  showing that $4$ matchings of size $2$ do not necessarily have a rainbow matching of size $2$. In \cite{alon, sudakov} bounds are studied in the integral case, in particular showing a  lower bound exponential in $r$. 
 
 Kotlar and Ziv proved a matroidal generalization of Theorem \ref{t.drisko}:

\begin{theorem}[Kotlar and Ziv \cite{KZ}]\label{t.kz}
Let $\M_1, \M_2$ be two matroids on the same vertex set $V$. Then any $2k-1$ sets $E_1, E_2, \dots, E_{2k-1}$ of size $k$ in $\M_1 \cap \M_2$ have a rainbow set of size $k$ belonging to $\M_1 \cap \M_2$.
\end{theorem}

Theorem \ref{t.drisko}  is obtained by taking $\M_1$ and $\M_2$ to be the two partition matroids whose parts are (respectively) the stars in the two sides of the bipartite graph.   

The aim of this paper is to prove a matroidal generalization of the $r$-partite case of Theorem  \ref{t.hypermatchings}, along the lines of Theorem \ref{t.kz}. By way of apology, most of the ideas are not new: the course of the proof follows closely that of Theorem \ref{t.hypermatchings}. But there are points where the matroidal version poses its peculiar difficulties. In particular, in order for a perturbation argument used in the proof of Theorem \ref{t.hypermatchings} to be adapted to the matroidal case, we need to  invoke  some properties of matroids and of submodular functions. These appear in Lemmas  \ref{l.manypds}, \ref{chain},   \ref{l.unionclosed},  and in Theorem \ref{l.inequality}. These are possibly of some independent interest.

To formulate the main result, we need a matroidal generalization of the notion of fractional matchings. This involves the familiar notion of  \emph{matroid polytopes}. 
For a function $\vx$ on a set $V$ and a subset $A$ of $V$, let $\vx[A]=\sum_{a
\in A}\vx(a)$. We denote the total size of $\vx$, namely $\vx[V]$, by $|\vx|$.
 
\begin{definition}\cite{SCHRIJVER}
Let $\M$ be  a matroid on a ground set $V$. 
The polytope of $\M$, denoted by $P(\M)$,
is $$\{\vx \in \R_+^V \mid  \vx[A] \le \rank_\M A ~\text{ for~ every}~ A \subseteq V\}.$$
 \end{definition} 
 
 Edmonds \cite{ed2} proved that all vertices of $P(\M)$ are integral, and that this is true also for the intersection of two matroids. 
 
 \begin{theorem}[\cite{SCHRIJVER}] \label{2matroids}
 If $\M_1, \M_2$ are matroids on the same ground set, then the vertices of the polytope $P(\M_1) \cap P(\M_2)$ are integral.
 \end{theorem}
   This is a corollary of another theorem of Edmonds, the two matroids intersection theorem \cite{ed2}. 

Our main result is:

\begin{theorem}\label{t.main}
Let $\M_1, \ldots ,\M_r$
be matroids on the same ground set $V$, and let $k$ be a real number. Let $d=r\lceil k\rceil -r+1$. Let $f_1, \ldots ,f_d$ be non-negative real valued functions belonging to $\bigcap_{i\le r}P(\M_i)$, satisfying  $|f_j|\ge k$ for every $j \le d$. Let $F_i=supp(f_i), ~i \le d$. Then there exists a function  $f \in \bigcap_{i\le r}P(\M_i)$ such that $\supp(f)$ is  a rainbow set of $(F_1, \ldots,F_d)$, and $|f| \ge k$.
\end{theorem}

 Theorem \ref{t.kz} follows. Let $E_i,~ i \le 2k-1$ be sets as in that theorem. Applying Theorem \ref{t.main} to the functions $\chi_{E_i}, ~i \le 2k-1$ (here and below $\chi_S$ is the characteristic function of the set $S$), yields a function  $f \in P(\M_1) \cap P(\M_2)$ with $|f| \ge k$ whose support is a rainbow set for the $E_i$'s. The function $f$ is a convex combination of vertices of  $P(\M_1) \cap P(\M_2)$, and since in this combination all coefficients are positive, the supports of these vertices are contained in $\supp(f)$. Among these there is at least one vertex $g$ with $|g| \ge |f|$. By Theorem  \ref{2matroids} $g$ is integral, namely a $0,1$ function, meaning that it is a characteristic function of a set as  in the conclusion of Theorem \ref{t.kz}.

  To obtain the $r$-partite case of Theorem \ref{t.hypermatchings} from  Theorem  \ref{t.main}, choose the matroids $\M_i, ~i \le r$ to be the partition matroids on $\bigcup_{i \le d}E_i$ defined by the stars in the $i$-th side $V_i$ of the hypergraph. Namely, a set is independent in $\M_i$ if it does not contain two edges meeting in $V_i$. Then a function belongs to $\bigcap_i P(\M_i)$ if and only if  it is a  fractional matching. The condition  $\nu^*(E_j) \geq k$ means that there exists a fractional matching $f_j \in \bigcap_i P(\M_i)$ with $\supp(f_j) \subseteq E_j$ and $|f_j| \ge k$ ($j \le d$).   
Applying Theorem \ref{t.main} then yields a fractional matching $f$ whose support is rainbow with respect to the sets $E_j$.

\vspace{1.5mm}




\section{A Topological Tool}

A {\em complex} is a downward-closed collection of sets, that in this context are called {\em faces}. Let $\C$ be a complex on a vertex set $V$. A face $\sigma$ of $\C$  is called a {\em collapsor} if it is contained in a unique maximal face. The operation of removing from $\C$ all faces containing a collapsor $\sigma$ is then called a {\em collapse}, and if $|\sigma| \le d$ then it is called
a $d$-{\em collapse}. We say that $\C$ is \emph{$d$-collapsible} if it can be reduced to $\emptyset$ by a  sequence of  $d$-collapses.
Wegner~\cite{Wegner} observed that a $d$-collapsible complex is \emph{$d$-Leray}, meaning that the homology groups of all induced complexes vanish in dimensions $d$ and higher. 


Our main tool will be a theorem of Kalai and Meshulam \cite{KM}. For a complex $\C$ let  $\C^c$ be the collection of all non-$\C$-faces (namely, $\C^c:=2^V \setminus \C$).

\begin{theorem}[Kalai-Meshulam \cite{KM}]\label{t.km}
If $\C$ is $d$-collapsible, then every $d+1$ sets in $\C^c$ have a rainbow set belonging to $\C^c$.
\end{theorem}

In fact, this is a special case of the main theorem in \cite{KM}. The way to derive it from the original theorem can be found in \cite{AHJ}.

We will use Theorem \ref{t.km} to reduce Theorem \ref{t.main} to a topological statement. To state this, we first extend the definition of the fractional matching number $\nu^*$ to our matroidal setting. For each $W \subseteq V$,  let 
\[
\nu^*(W):=\max \bigg\{|\vx|: \vx \in \bigcap_i P(\M_i), \; \supp(\vx) \subseteq W \bigg\}.
\]
For a positive real $k$ let $\cx_k$ be the simplicial complex of all sets $W \subseteq V$ with $\nu^*(W)<k$.

\begin{theorem}\label{t.specific}

$\cx_{k}$ is $(r \left\lceil k\right\rceil -r)$-collapsible.
\end{theorem}

 Theorem \ref{t.main} follows from Theorem \ref{t.specific}. Indeed, as $\cx_k$ is $(r\left\lceil k\right\rceil -r)$-collapsible, by Theorem \ref{t.km} any $r\left\lceil k\right\rceil -r  +1$ sets not in $\cx_k$ contain a rainbow set not in $\cx_k$. Since $F \not\in \cx_k$ means that some $f \in \bigcap_{i \leq r} P(\M_i)$ supported  on $F$ satisfies $|f| \geq k$, Theorem \ref{t.main} follows.

\bigskip
\section{Proof of Theorem \ref{t.specific}}


 
A non-negative function $\c: 2^V \rightarrow \R_+$ is said to be  \emph{decreasing} if $c(A) \leq c(A')$ whenever $A \supseteq A'$.
A non-negative function $\c: 2^V \rightarrow \R_+$ is said to be  \emph{submodular} if, whenever $A, B \subseteq V$, we have 
\[
c(A) + c(B) \geq c(A \cup B) + c(A \cap B).
\]
Note that the rank function $\rk_\M$ of a matroid $\M$ is submodular \cite{Welsh}.


\begin{definition}
If $\c: 2^V \rightarrow \R_+$ is decreasing  and $\M$ is a matroid on $V$, let   
$$P_\c(\M):=\{\vx \in \R_+^V \mid  \vx[A] \le c(A) \rank_\M A ~\text{ for every}~ \emptyset \neq A \subseteq V\}.$$


\end{definition}
Note that excluding the $A=\emptyset$ inequality does not change the polytope.

We shall use the acronym PDS for ``positive, decreasing and submodular''. As in \cite{AHJ}, we shall consider perturbations of  $\cx_k$. For this purpose, we shall need the following:
\begin{lemma}\label{l.manypds}
The polytope $Q$ of PDS functions on $2^V$ has full dimension. Moreover, for any $b >0$, the polytope $Q \cap \{c(V)=b\}$ has full dimension (namely $2^{|V|}-1$) relative to the hyperplane $\{c(V) =b\}$, for any $b > 0$.
\end{lemma}
\begin{proof}
To show the first claim, let
$c(A):=2|V|^2-|A|^2$ for every $A \subseteq V$. We claim that $\c \in interior(Q)$. Clearly, $\c$ is  strictly positive and strictly decreasing. 
To show strict submodularity, 
note that if $A \neq B \subseteq V$ then
\begin{align*}
& c(A) + c(B) - c(A \cup B)- c(A \cap B)\\
=& |A \cup B|^2 + |A \cap B|^2 - |A|^2 - |B|^2 \\
=& \frac{1}{2}(|A \cup B|-|A \cap B|)^2 + 
\frac{1}{2}(|A|-|B|)^2 
 >0,
\end{align*}
(To obtain  the second equality we subtracted from both sides of the equation $\frac{1}{2} \big((|A \cup B| + |A \cap B|)^2 - (|A| + |B|)^2\big)=0$). It is not necessary to check the case  $A=B$, since in this case equality is true for any function. 

To show the second claim, let $\c':=\frac{b}{|V|^2}\c$ for $\c$ as above. Then $\c'$  maintains the strictness of all inequalities defining $Q$, and satisfies $\c'(V)=b$. 
\end{proof}

Given  an $r$-tuple $\b=(\b^1, \dots, \b^r)$  of PDS functions on $2^V$ 
and a non-negative vector $\a=(a_v)_{v \in V}$,  let $\nu^*_{\a,\b}(W)$ be the largest possible value of $\a \cdot f$ among all $f \in \bigcap P_{\b^i}(\M_i)$ with $\supp(f) \subseteq W$.
That is:
\[
\begin{array}{rlll}
        \nu^*_{\a,\b}(W) := 
        \max & \sum\limits_{v \in W} a_v  f(v) & & \\
        \text{ s.t. } & \sum\limits_{v \in A}  f(v) & \leq  b^i(A)\rk_{\M_i}(A)
        & \forall A \subseteq V, \forall i \in [r], \\
        & \text{ and }f(v) & \geq 0
        &  \forall v \in W.
\end{array}
\]

By linear programming duality, $\nu^*_{\a,\b}(W)$ is equal to
\[
\begin{array}{rlll}
        \tau^*_{\a,\b}(W) := 
        \min & \sum\limits_{\substack{i \in [r] \\ A \subseteq V}} b^i(A) \rank_{\M_i}(A)  h(i,A) & \\
        \text{ s.t. } & \sum\limits_{\substack{i \in [r] \\ A \ni v}} h(i, A) \geq  a_v 
        & \forall v \in W \\
        & h(i,A) \geq 0
        &  \forall A \subseteq V, i \in [r]
\end{array}
\]

Given a positive real number $k$, let  $\cx_{\a,\b,k}$ be the simplicial complex consisting of all sets $W \subseteq V$ for which $\nu_{\a, \b}^*(W) < k$.

\begin{theorem}\label{t.general}
Let $\a \in \mathbb{R}_+^V$ and let $\b$ be an $r$-tuple of PDS functions on $2^V$. Let $\ba = \min_{ V} \{a_v\}$,  $\bb =  \min_{i \in [r]} \{b^i(V)\}$.
Then $\cx_{\a,\b,k}$ is $r \left\lfloor \frac{\bk}{\ba \bb} \right\rfloor$-collapsible, where $\bk$
is given by
\[ \bk:= \max\{ \vab(W): W \in \cx_{\a,\b,k} \} < k.
\]
\end{theorem}

Theorem \ref{t.specific} is the special case of Theorem \ref{t.general} obtained by fixing every $b^i(A)=1$ and $\a=\one$. Theorem \ref{t.general} applies since the constant-1 function is  PDS.
Here, $\cx_k = \cx_{\a,\b,k}$,
$\ba=\bb=1$,
and $\left\lfloor \bk \right\rfloor \leq \left\lceil k \right\rceil -1$,
yielding that $\cx_k$ is $(r \left\lceil k \right\rceil-r)$-collapsible. 


\vspace{1mm}

We prove Theorem \ref{t.general} by induction on $|\cx_{\a,\b,k}|$. Note  that  $|\cx_{\a,\b,k}| > 1$ , since $\cx_{\a,\b,k}$ contains at least one nonempty set.

Following a crucial idea from \cite{AHJ}, we may assume that generically, for every $W \subseteq V$  there is a unique function $h$ on $[r] \times 2^V$ attaining the minimum in the program defining $\tau^*_{\a,\b}(W)$. For, the set of all $\b=(\b^1, \dots, \b^r)$ for which the optimum is not uniquely attained is the  union of finitely many hyperplanes. By Lemma \ref{l.manypds}, it is possible to perturb the $\b^i$'s so as to avoid these hyperplanes, in a fashion sustaining the value of $\bb$. If the perturbation is sufficiently small, $\cx_{\a,\b,k}$ stays unaffected.

Now, we choose any $W \in \cx_{\a,\b, k}$ such that:
\[
\nu_{\a,\b}^*(W) = \bk, 
\text{ and } W \text{ is inclusion-minimal among all such sets.}
\tag{$\dagger$}
\]

We prove that removing all supersets of $W$ is an elementary $r\left\lfloor \frac{\bk}{\ba \bb} \right\rfloor$-collapse in $\cx_{\a,\b,k}$. This requires the three claims $(\diamondsuit), (\clubsuit)$, and $(\spadesuit)$ as follows, which together will constitute the remainder of the proof of Theorem \ref{t.general}.
\[
 W \text{ is contained in a unique facet. }
 \tag{$\diamondsuit$}
 \]

To prove $(\diamondsuit)$, we follow \cite{AHJ}, but reproduce the argument for completeness. Let $W^+:= \{ v \in V: W \cup \{v\} \in  \cx_{\a,\b,k} \}$. Let $v \in W^+$ be arbitrary. By maximality of $\bk$, we know $\nu^*_{\a,\b} (W \cup \{v\}) = \vab(W) = \bk$, and hence   
$\tab(W \cup\{v\}) = \tab(W)=\bk$. By our assumed perturbations, there exists a unique function $h$ on $[r] \times 2^V$ attaining the minimum defining $\tab(W)$.
Since the function $h'$ witnessing $\tab(W \cup \{v\})=\bk$ is also feasible for $\tab(W)$, it follows that $h'=h$, so $h$ must satisfy the additional constraint $\sum_{i \in [r], A \ni v} h(i,A) \geq a_v$ for  $v$. Since this is true for every $v \in W^+$, the function  $h$ satisfies the constraints for all vertices in $W \cup W^+$, witnessing $\tab(W \cup W^+) = \bk$. Thus $W \cup W^+ \in \cx_{\a,\b,k}$ is the unique facet containing $W$, giving $(\diamondsuit)$.

\[
\text{ If } W \text{ satisfies } (\dagger) \text{ then }  |W| \leq r\left\lfloor \frac{\bk}{\ba \bb} \right\rfloor.
 \tag{$\clubsuit$}
\]

The proof of $(\clubsuit)$ is the main place where new arguments are needed, beyond those appearing in  \cite{AHJ}. These appear in Lemma \ref{chain}, Theorem \ref{l.inequality} and Lemma  \ref{l.unionclosed} below. 
Let $P_W$ be the polytope of functions $f$ on $\R^W$ satisfying $f(v) \geq 0$ for all $v \in V$, and $\sum_{v \in A} f(v) \leq b^i(A)\rk_{\M_i} (A)$ for all $i \in [r]$ and $A \subseteq V$.
Let  $f$ be a vertex of $P_W$ at which the maximum value of $\sum_{v \in W} a_v f(v)$ is attained. This maximum is at least $\bk$. Then
$f$ must satisfy $|W|$ linearly independent inequalities of the above kinds at equality. If $f(v)=0$ were true for any $v$, then $f$ would also witness $\vab(W \setminus \{v\})=\vab(W) = \bk$, contradicting minimality of $W$. So all $|W|$ equalities are of the form 
$\sum_{v \in A} f(v) = b^i(A)\rk_{\M_i} (A)$. For each $i \in [r]$ let $w_i$ be the number of equalities of the form
$\sum_{A} f(v) = b^i(A)\rk_{\M_i} (A)$  (so $\sum_i w_i=|W|$). 

Let
\[
\F^f_i:= \Big\{A \subseteq V: \sum_{v \in A} f(v) = b^i(A)\rk_{\M_i} (A) \Big\},
\]
so the set $\{\chi_A \mid A \in \F^f_i \}$ consists of  $w_i$ linearly independent vectors.

We can take advantage of these $w_i$ sets as follows.  Recall that $\chi_S$ denotes the indicator vector of $S$. We use the term ``\emph{chain} of length $r$ of sets'' for a collection of $r$ distinct non-empty sets, totally ordered by inclusion.

\begin{lemma}\label{chain}
Let $\cf \subseteq 2^{[n]}$ be a  family of sets, closed under intersections and unions.
If  $\{ \chi_S: S \in \cf \}$ linearly spans (over the reals or rationals) a space of dimension $t$, then $\cf$ contains a chain of length $t$.
\end{lemma}

\begin{proof}
We proceed by induction on $t$.
It is obvious when $t = 1$.
For $t \geq 2$, we may assume that there exists a chain $\emptyset \neq A_1 \subsetneq \dots \subsetneq A_{t-1}$ of length $t-1$.
Since $\{ \chi_S: S \in \cf \}$ spans a $t$-dimensional space, there exists a non-empty set $A \in \cf$ such that $\chi_A \not\in U:= \text{span}(\{ \chi_{A_i}: i<t \})$.
If $A \not\subseteq A_{t-1}$, then letting $A_t=A_{t-1} \cup A$ yields the desired chain of length $t$.  
Thus we may assume $A \subseteq A_{t-1}$.

For $i=2, \dots, t$ let $B_i = A_i \setminus A_{i-1}$ and let $B_1 = A_1$.
Note that $\chi_{B_i} = \chi_{A_i}-\chi_{A_{i-1}} \in \text{span}(\{ \chi_{A_i}: i<t \})$.
If for some $i \leq t$ neither $B_i \cap A = \emptyset$ nor $B_i \subseteq A$, then $(A \cup A_{i-1}) \cap A_i \in \cf$ lies strictly between $A_{i-1}$ and $A_i$, so its addition forms the desired chain. We may thus assume that there is no such $B_i$.

Let $S=\{i \leq t : B_i \subseteq A\}$. By the above assumption $A = \bigsqcup\limits_{i \in S} B_i$. Hence
$\chi_A = \sum_{i \in S} \chi_{B_i} \in U$, a contradiction.
\end{proof}

We wish to show that each $\F^f_i$ 
satisfies the condition of Lemma \ref{chain}, namely it is closed under intersections and unions. Indeed, for the usual matroid polytopes, it is a well-known fact (see Lemma \ref{l.unionclosed} below).
Extending this to skew polytopes first requires the following result.
\begin{theorem}\label{l.inequality}
If $\c,\r$ are nonnegative submodular functions on a lattice of sets, $\c$ is decreasing and $\r$ is increasing, then  $\c \cdot \r$ is submodular.
\end{theorem}

This may be folklore, and it closely resembles a standard fact on the product of convex functions (see e.g.\ \cite[3.32]{BL}), but Lov\'asz's celebrated method \cite{Lo} for linearly extending submodularity to convexity does not behave well under taking products, and so we could not establish a direct implication. The only explicit reference we found is a question answered at \cite{stackexchange}.
For completeness we provide a proof here.
\begin{proof}

We wish to show that, for any $A,B \subseteq V$,
\[
c(A \cup B) r(A \cup B) + c(A \cap B) r(A \cap B) - c(A) r(A) - c(B) r(B) \leq 0.
\]

For a real-valued function $h$ on a lattice,
let
$D_{T}h(S):=h(S \cup T) - h(S)$ be  the ``difference'' operator applied to $h$.
In this terminology, a function $h$ is submodular if and only if
\[
D_{B \setminus A} D_{A \setminus B} (h) (A \cap B) \leq 0.
\]
We shall show that $D_S D_R (cr)$ is non-positive for any sets $S,R$. To see this, write:
\begin{align*}
c(S \cup R)r(S \cup R) -c(S)r(S) &=c(S \cup R) (r(S \cup R) -r(S)) \\ &+ (c(S\cup R)-c(S))r(S)    
\end{align*}
gives us the product rule
$D_R(cr)(S)=c(S \cup R)D_Rr(S) + (D_R c(S))(r(S))$. Letting $T_R h(X)$ denote $h(X \cup R)$ for any $h$, this says
\[
D_R(cr)= (T_R c) (D_Rr) + (D_R c)r.
\]
 Applying this  twice gives 
\begin{align*}
    D_S D_R(cr) = &
    D_S \big( (T_R c )( D_R r) \big)
    +  D_S \big( (D_R c)(r)
    \big) \\
    = &
    T_S T_R c \cdot D_S D_R r
    + \big( D_S T_R c \big) \cdot
    D_Rr 
     \\
    & + \big(
    D_R T_S c 
    \big) \cdot 
    D_S r +
    \big( D_S D_R c \big) \cdot r.
\end{align*}
All four products above are non-positive, as can be seen from the following:
\begin{itemize}
    \item $c,r \geq 0$ by nonnegativity,
    
    \smallskip 
    
    \item $D_R D_T r, D_R D_T c \leq 0$ by submodularity,
    
    \smallskip 
    
    \item $D_R c, D_T c \leq 0$ as $\c$ decreasing,
    
    \smallskip 
    
    \item $D_T r, D_R r \geq 0$ as $\r$ increasing.
\end{itemize}
\end{proof}

\begin{lemma}\label{l.unionclosed}
Let $\M$ be a matroid on $V$, $\c$ a PDS function on $2^V$, $f$ a point in $P_\c(\M)$, and $W$ a subset of $V$. Let $\cf$
be the family of all subsets $A$ of $W$ satisfying 

\begin{equation} \label{cr} \sum_{v \in A} f(v) = c(A) \rk (A). \end{equation} 

Then $\cf$ is closed under intersections and unions.
\end{lemma}

\begin{proof}
Let $A,B \in \cf$, so $\sum_{A} f(v) = c(A) \rk (A)$
and $ \sum_{B} f(v) = c(B) \rk (B)$.
Then
\begin{align*}
\sum_{v \in A \cup B} f(v)& \leq c(A \cup B)\rk (A \cup B) \\
&\leq c(A) \rk (A) + c(B) \rk (B) - c(A \cap B)\rk (A \cap B) \\
& = \sum_{v \in A} f(v) + \sum_{v \in B} f(v) -
c(A \cap B)\rk (A \cap B) \\
& \leq \sum_{v \in A} f(v) + \sum_{v \in B} f(v) -\sum_{v \in A \cap B} f(v) \\
& = \sum_{v \in A \cup B} f(v).
\end{align*}
The second inequality is  the submodularity of $\c \cdot \rk$.
The first and last inequalities follow from  the fact that $f \in P_\c(\M)$.
Since equality should hold throughout, it follows that $A \cup B, A \cap B \in \cf$.
\end{proof}

Lemma \ref{l.unionclosed} enables  application of Lemma \ref{chain} to $\cf:=\F^f_i$. We obtain a chain $\emptyset \neq A_1 \subsetneq A_2 \subsetneq \dots \subsetneq A_{w_i}$ in $\F^f_i$.
Thus
\[
0 < \sum_{v \in A_1} f(v) < \sum_{v \in A_2} f(v) < 
\dots < \sum_{v \in A_{w_i}} f(v) 
\]
as $f(v)>0$ for each $v \in W$.
We may rewrite this as 
\[
0<b^i(A_1)\rk_{\M_i}(A_{1})<
\dots <
b^i(A_{w_i})\rk_{\M_i}(A_{w_i}).
\]
But $\b^i$ is decreasing, so $b^i(A_1) \geq b^i(A_2) \geq \dots \geq b^i(A_{w_i})$. 
Thus 
\[
0<\rk_{\M_i}(A_{1})<
 \rk_{\M_i}(A_{2})<
\dots <
\rk_{\M_i}(A_{w_i}),
\]

Since  ranks are integers, it follows that  $\rk_{\M_i}(A_{w_i}) \geq w_i$. 

Thus in fact, for each $i \in [r]$:
\[
\ba \bb w_i \leq \ba b^i(A_{w_i}) \rk_{\M_i}(A_{w_i})
=  \ba \sum_{v \in A_{w_i}} f(v)
\leq \sum_{v \in W} a_v f(v)
= \bk,
\]
and by integrality $w_i \leq \left\lfloor \frac{\bk}{\ba \bb} \right\rfloor$.
So we conclude
\[
|W| = \sum_{i \in [r]} w_i \leq \sum_{i \in [r]}
\left\lfloor \frac{\bk}{\ba \bb} \right\rfloor
= r\left\lfloor \frac{\bk}{\ba \bb} \right\rfloor,
\]
which proves $(\clubsuit)$.

$(\spadesuit)$ Suppose $W$ satisfies $(\dagger)$ and let $\cx'$ be the complex obtained by removing from $\cx_{\a,\b,k}$ all faces containing $W$. Then   there exists  $\a' \in \mathbb{R}_+^V$, satisfying 
$r\left\lfloor \frac{\bk}{\ba' \bb} \right\rfloor
\leq r\left\lfloor \frac{k}{\ba \bb} \right\rfloor$, for which
\[
\cx' = \{W' \subseteq V: \nu^*_{\a',\b} (W') < \bk \} = \cx_{\a',\b,\bk}.
\]

The proof of  $(\spadesuit)$  follows a parallel argument in \cite{AHJ}. We claim that there is some $\epsilon > 0$ for which $\cx' = \cx_{\a',\b,\bk}$ is satisfied by the objective coefficients $\a'$ defined coordinate-wise by:

\[
a'_v := \left\lbrace
\begin{array}{lc}
    a_v - \epsilon & \text{ if } v \not\in W, \\
    a_v & \text{ if } v \in W.
\end{array}
\right.
\]

First consider any $W' \subseteq V$ that wasn't even in $\cx_{\a,\b, k}$ to begin with, so that $\vab (W') \geq k$.
The feasibility regions for $\vab (W')$ and $\nu^*_{\a',\b} (W')$ are the same, so if $\epsilon$ is sufficiently small relative to $k-\bk$, it follows $\nu^*_{\a',\b} (W') \geq \bk$, so that $W' \not\in \cx_{\a', \b, \bk}$ either.

Next, pick any $W' \subseteq V$ previously in $\cx_{\a,\b,k}$, but which contained $W$ so was removed in the collapse. As before, let $f$ be an optimiser for the LP defining $\vab(W)$, so $\a \cdot f = \bk$ but also $\supp(f) = W \subseteq W'$. This way, $f$ is also feasible for the linear program defining $\vab(W')$. But whenever $a'_v < a_v$, $e \not\in W$ and hence $ f(v)=0$ by minimality of $W$. Hence 
$\nu^*_{\a',\b}(W') \geq \a' \cdot f = \a \cdot f  = \vab(W') = \bk$. Thus $W' \not\in \cx_{\a', \b, \bk}$.

Finally, take some $W' \subseteq V$ previously in $\cx_{\a,\b,k}$ and not fully containing $W$. Note $W \cap W' \subsetneq W$.
We wish to show $\nu^*_{\a',\b}(W') < \bk$ for deducing $W' \in \cx_{\a',\b,\bk}$, so assume for contradiction $\nu^*_{\a',\b}(W') \geq \bk$, as witnessed by some
$g \in \bigcap P_{\b^i}(\M_i), \supp(g) \subseteq W'$ with $\a' \cdot g \geq \bk$.
We cannot have $\supp(g) \subseteq W \cap W'$. For otherwise
$g$ would also witness $\nu^*_{\a,\b}(W \cap W') 
\geq \a' \cdot g
= \a \cdot g
\geq \bk$, hence $\vab(W \cap W') =\bk$ by maximality of $\bk$, and this would contradict inclusion-minimality of $W$.
So there is at least one $e_0\in \supp(g) \setminus W$. So $g(v_0)>0$ and $a_{v_0}' < a_{v_0}$ means
$\sum_{v\in W'} a_v g(v)> \sum_{v \in W'} a'_v g(v)\geq \bk$, still contradicting maximality of $\bk$.

So, by inductive hypothesis, $\cx_{\a',\b,\bk}$ is indeed $r\left\lfloor \frac{\bk}{\ba' \bb} \right\rfloor$-collapsible, and since $\bk<k$, we can make $\epsilon$ small enough to guarantee $r\left\lfloor \frac{\bk}{\ba' \bb} \right\rfloor
\leq r\left\lfloor \frac{k}{\ba \bb} \right\rfloor$.

\section{Closing Remarks}

Theorem \ref{t.main} provides a matroidal generalisation of Theorem \ref{t.hypermatchings}. While the proof method goes via a weighted version, Theorem \ref{t.general}, this does not seem to also generalise the corresponding theorem for weighted fractional matchings (see \cite[Theorem 3.2]{AHJ} for the non-$r$-partite version).

Indeed, suppose we are given an $r$-partite hypergraph $H$ with parts $V=V_1 \cup \dots \cup V_r$, along with vertex weights $\{b_v:v \in V\}$.
We wish to define a collection of polytopes $\{P_i: i \in [r]\}$ with the following property. A weighted collection of edges $\{x_e :e \in E(H)\}$ is a fractional matching with respect to the $\{b_v\}$'s if and only if $x \in \bigcap_{i \in [r]} P_i$. To do so in a way that would generalise the usual (non-weighted) case would suggest we let
\[
P_i := \{x \in \R_+^{E(H)}: \forall A \subset E(H), x[A] \leq b^i(A) |(\cup A) \cap V_i|\},
\]
for some $b^i:2^{E(H)} \rightarrow \R_+$ satisfying $b^i(N(v)) = b_v$ for every $v \in V_i$ (where $N(v)$ denotes all edges of $H$ incident to $v$). This can be done by letting $b^i(A):=\max \{ b_v: v \in  (\cup A) \cap V_i \}$-this way, all inequalities not of the form $A=N(v)$ for some $v \in V_i$ are redundant.
But while these $b^i$'s are submodular, they are not decreasing, and hence Theorem \ref{t.general} does not apply.

It is therefore natural to ask what is the largest family of functions for which Theorem \ref{t.general} holds.

\medskip
\section*{Acknowledgements}
We are indebted to Ron Aharoni for introducing us to this problem and guiding helpful discussions.

\end{document}